\documentclass[11pt]{amsart}

\usepackage{amsmath,amssymb,amsthm}
\usepackage{a4wide}
\usepackage{hyperref}

\usepackage{graphicx}
\usepackage{xypic}
\entrymodifiers={+!!<0pt,\fontdimen22\textfont2>}
\usepackage[all]{xy}

\newtheoremstyle{myremark} 
    {7pt}                    
    {7pt}                    
    {}  	                 
    {}                           
    {\bf}       	         
    {.}                          
    {.5em}                       
    {}  

\theoremstyle{plain}
\newtheorem{lemma}{Lemma}[section]
\newtheorem{theorem}[lemma]{Theorem}
\newtheorem{fact}[lemma]{Fact}
\newtheorem{definition}[lemma]{Definition}
\newtheorem{corollary}[lemma]{Corollary}
\newtheorem{proposition}[lemma]{Proposition}
\newtheorem{conjecture}[lemma]{Conjecture}
\newtheorem{problem}[lemma]{Problem}
\newtheorem{claim}{Claim}

\theoremstyle{myremark}
\newtheorem{remark}[lemma]{Remark}

\newtheorem{notdef}[lemma]{Notation/Definition}

\newcommand{\nat}{\mathbb{N}}

\newcommand{\er}{\mathbb{R}}

\newcommand{\cl}{\mathrm{Cl}}

\renewcommand{\subset}{\subseteq}

\newcommand{\lk}{\mathrm{lk}}

\newcommand{\Michal}[1]{}
\newcommand{\Honza}[1]{}
\newcommand{\remove}[1]{}

\begin{document}
\title[An upper bound theorem for a class of flag weak pseudomanifolds]{An upper bound theorem for a class of flag weak pseudomanifolds}

\author[Micha{\l} Adamaszek]{Micha{\l} Adamaszek}
\address{ALTA and Fachbereich Mathematik, Universit\"at Bremen
      \newline Bibliothekstr. 1, 28359 Bremen, Germany}
\email{aszek@mimuw.edu.pl}
\thanks{Research supported by a DFG grant.}


\begin{abstract}
If $K$ is an odd-dimensional flag closed manifold, flag generalized homology sphere or a more general flag weak pseudomanifold with sufficiently many vertices, then the maximal number of edges in $K$ is achieved by the balanced join of cycles. 

The proof relies on stability results from extremal graph theory. In the case of manifolds we also offer an alternative (very) short proof utilizing the non-embeddability theorem of Flores. 

The main theorem can also be interpreted without the topological contents as a graph-theoretic extremal result about a class of graphs such that 1) every maximal clique in the graph has size $d+1$ and 2) every clique of size $d$ belongs to exactly two maximal cliques.
\end{abstract}
\maketitle

\section{Introduction}
\label{section:intro}

If $K$ is a finite simplicial complex of dimension $d$ we write
$$(f_{-1}(K),f_0(K),\ldots,f_d(K))$$
for its face vector, or $f$-vector, meaning that $f_i(K)$ is the number of $i$-dimensional faces of $K$. A classical problem in enumerative combinatorics is to characterize the $f$-vectors of interesting classes of simplicial complexes.

Our work is motivated by this kind of a question for \emph{flag spheres}, although the methods we use work for a larger class of flag weak pseudomanifolds, see Definition~\ref{def:weak}. A simplicial complex is called \emph{flag} if every minimal non-face of $K$ has two vertices. This means that the faces of $K$ are precisely the cliques in the graph $G=K^{(1)}$, the $1$-skeleton of $K$. In this case we write $K=\cl(G)$ and say $K$ is the \emph{clique complex} of $G$. The interest in the $f$-vectors of flag spheres is inspired by the Charney-Davis conjecture \cite{CD} and its generalizations \cite{Gal,NevoPet,NevoPetTenner,MurNevo}.

The properties of $f$-vectors of flag complexes can be quite different from those of arbitrary simplicial complexes. For example, a simplicial $(2s-1)$-sphere $K$ satisfies
$$2sf_0(K)-s(2s+1)\leq f_1(K)\leq \frac12 f_0(K)^2-\frac12 f_0(K)$$
with both inequalities best possible. On the other hand, for a \emph{flag} simplicial $(2s-1)$-sphere the corresponding inequalities ought to be
\begin{equation}\label{pupcia}(4s-3)f_0(K)-8s(s-1)\leq f_1(K)\leq \frac{s-1}{2s} f_0(K)^2+f_0(K).\end{equation}
These are known to hold when $s=1$ (obvious) and $s=2$ (\cite{DOkun,Gal}) while for $s\geq 3$ both upper and lower bounds have the status of conjectures \cite{Gal, NevoPet}. One consequence of our main result is that the upper bound in \eqref{pupcia} holds for flag $(2s-1)$-spheres for any $s$, provided that the number of vertices $f_0(K)$ is sufficiently large.

Our main result, Theorem~\ref{thm:weakpseudo}, holds for a class of complexes which we now define.

\begin{definition}
\label{def:weak}
A \emph{$d$-dimensional weak pseudomanifold} is a pure simplicial complex of dimension $d$ such that every $(d-1)$-dimensional face belongs to exactly two faces of dimension $d$.
\end{definition}

Note that the only $0$-dimensional weak pseudomanifold is the complex $S^0$ with two isolated vertices and that every $1$-dimensional weak pseudomanifold is a disjoint union of cycles. Moreover, if $K$ is a $d$-dimensional weak pseudomanifold and $\sigma$ is any face then $\lk_K\sigma$ a weak pseudomanifold of dimension $d-|\sigma|$. Finally, note that the join of weak pseudomanifolds of dimensions $d$ and $d'$ is a weak pseudomanifold of dimension $d+d'+1$.

We can now state the main theorem.

\begin{theorem}
\label{thm:weakpseudo}
For every $s\geq 1$ and every constant $C>0$ there is an $n_0=n_0(s,C)$ with the following property. If $K$ is a $(2s-1)$-dimensional flag weak pseudomanifold with $f_0(K)\geq n_0$ and $f_s(K)\leq Cf_0(K)^s$, then
$$f_1(K)\leq \frac{s-1}{2s}f_0(K)^2+f_0(K).$$
\end{theorem}
The upper bound is clearly best possible, since equality holds if $K$ is the join of $s$ copies of the $\frac{1}{s}f_0(K)$-vertex cycle.\footnote{In general it should be understood that the factors of the join are as balanced as possible, i.e. all have size $\lfloor\frac{f_0(K)}{s}\rfloor$ or $\lceil\frac{f_0(K)}{s}\rceil$. Here and in the future we allow ourselves this level of language imprecision in informal or asymptotic statements.}

\medskip
The statement of the theorem may seem a bit technical, but it implies upper bounds for some quite concrete classes of complexes, as we now explain. 

Recall that if $\sigma\in K$ is a face then the \emph{link} of $\sigma$, denoted $\lk_K\sigma$ is the subcomplex $\{\tau\in K~|~\tau\cap\sigma=\emptyset,\ \tau\cup\sigma\in K\}$. A simplicial complex is \emph{pure} if all maximal faces have the same dimension. A pure $d$-dimensional complex $K$ is a \emph{simplicial $d$-sphere} if for every face $\sigma$ the link $\lk_K\sigma$ is homeomorphic to the $(d-|\sigma|)$-dimensional sphere and it is a \emph{simplicial manifold} (closed, compact) if that condition holds for all $\sigma\neq\emptyset$. Further, we say $K$ is \emph{Gorenstein$^*$}, resp. \emph{Eulerian}, if for every face $\sigma$ the link $\lk_K\sigma$ has the same homology groups, resp. the same Euler characteristic, as the $(d-|\sigma|)$-sphere. Gorenstein$^*$ complexes are also called \emph{generalized homology spheres}. The relation between these classes of pure simplicial complexes is best explained by the diagram:
$$
\xymatrix@=1em{
2s\textrm{-sphere} \ar@2[d]\ar@2[r]& 2s\textrm{-manifold} \ar@2[dr]^{s\geq 1}& \\
\textrm{Gorenstein}^*\ar@2[r] & \textrm{Eulerian}\ar@2[r] & \textrm{weak pseudomanifold} \\
(2s-1)\textrm{-sphere} \ar@2[u]\ar@2[r] & (2s-1)\textrm{-manifold}\ar@2[u] & \\
}
$$
We will also need the $h$-vector $(h_0(K),\ldots,h_{d+1}(K))$ which is a linear transformation of the $f$-vector, defined by the polynomial identity $\sum_{i=0}^{d+1} h_i(K)x^{d+1-i}=\sum_{i=-1}^d f_{i}(K)(x-1)^{d-i}$. Any Eulerian complex $K$ satisfies the \emph{Dehn-Sommerville equations} $h_i(K)=h_{d+1-i}(K)$ for $i=0,\ldots,d+1$. 

\medskip
We can now present the more specific consequences of Theorem~\ref{thm:weakpseudo}.

\begin{corollary}
\label{thm:mainodd}
For every $s\geq 1$ the inequality
\begin{equation}
\label{eq:thmmain}
f_1(K)\leq \frac{s-1}{2s}f_0(K)^2+f_0(K)
\end{equation}
holds whenever $K$ is a sufficiently large
\begin{itemize}
\item[a)] $(2s-1)$-dimensional flag weak pseudomanifold which satisfies the middle Dehn-Sommerville equation $h_{s-1}(K)=h_{s+1}(K)$,
\item[b)] flag simplicial $(2s-1)$-manifold, 
\item[c)] flag $(2s-1)$-Gorenstein$^*$ complex.
\end{itemize}
\end{corollary}
\begin{proof}
Of course b) and c) follow from a), since every simplicial $(2s-1)$-manifold or Gorenstein$^*$ complex is Eulerian, hence it satisfies \emph{all} the Dehn-Sommerville equations. To prove a) rewrite the relation $h_{s-1}(K)=h_{s+1}(K)$ using the face numbers of $K$. It takes the form of a linear equation
$$f_s(K)=a_{s,s-1}f_{s-1}(K)+\ldots+a_{s,0}f_0(K)+a_{s,-1}f_{-1}(K)$$
where $a_{s,i}$ are some universal constants. Then we have
$$f_s(K)\leq\sum_{i=-1}^{s-1} |a_{s,i}|{f_0(K)\choose i+1}\leq \Big(\sum_{i=-1}^{s-1} |a_{s,i}|\Big)\cdot f_0(K)^s$$
hence we are in the situation of Theorem~\ref{thm:weakpseudo}.
\end{proof}

\subsection*{Related work} The inequality \eqref{eq:thmmain} is obvious for any $1$-dimensional weak pseudomanifold and  for $s=2$ it is known to hold without any restrictions on $f_0(K)$ if $K$ is a flag $3$-manifold or a flag $3$-Gorenstein$^*$ complex, see \cite{Gal}. In \cite{Idiots} the authors used similar methods to study the case $s=2$, but the purpose of that work was to verify a conjecture of \cite{Gal} that if $f_1(K)$ is sufficiently close to the extremal value then $K$ is a join of two cycles. Here we move from $s=2$ to arbitrary $s$, but we are only interested in the extremal value itself, although it seems reasonable to conjecture that a similar approximation result also holds. 

Finally if $K$ is $d$-dimensional and Eulerian then one can define the $\gamma$-vector $(\gamma_0(K),\ldots,\gamma_{s}(K))$ of $K$, where $s=\lfloor\frac{d+1}{2}\rfloor$, by the identity
$\sum_{i=0}^{d+1} h_i(K)x^i=\sum_{i=0}^s \gamma_{i}(K)x^i(x+1)^{d+1-2i}$
as in \cite{Gal}. It is conjectured that if $K$ is a flag $d$-Gorenstein$^*$ complex then the vector $\gamma(K)$ is non-negative and even more generally that it is an $f$-vector of some flag simplicial complex \cite{Gal,NevoPet}. The latter requires that $\gamma(K)$ satisfies the so-called Frankl-Furedi-Kalai inequalities \cite{Andy}, and the only ones which are known to hold in arbitrary dimension are $\gamma_0(K)=1$ and $\gamma_1(K)\geq 0$. Our result proves (for sufficiently large complexes of odd dimension $d=2s-1$) another inequality from that set, namely $\gamma_2(K)\leq\frac{s-1}{2s}\gamma_1(K)^2$, which is precisely equivalent to \eqref{eq:thmmain}.

\subsection*{A formulation in terms of graphs}
Since a flag complex is completely determined by its $1$-skeleton, Theorem~\ref{thm:weakpseudo} can be equivalently phrased in the language of graphs.

First we introduce some notation. If $G$ is a graph and $X\subset V(G)$ then $G[X]$ is the induced subgraph of $G$ with vertex set $X$. For any $v\in V(G)$ and $X\subset V(G)$ we define the neighborhood of $v$ as $N_G(v)=\{u~:~uv\in E(G)\}$ and we set $$\deg_G(v,X)=|X\cap N_G(v)|.$$ We will usually omit $G$ from the notation and write simply $\deg(v,X)$.

By $K^s(n_1,\ldots,n_s)$ we denote the complete $s$-partite graph with parts of sizes $n_1,\ldots,n_s$. By $k_s(G)$ we mean the number of $s$-element cliques in $G$. If $K=\cl(G)$ then of course $f_s(K)=k_{s+1}(G)$ for all $s$.

If $\sigma$ is a clique in $G$ then we define the \emph{link} of $\sigma$ as
$$\lk_G\sigma=G\left[\bigcap_{v\in\sigma} N_G(v)\right]\;.$$
In particular $\lk_Gv=G[N_G(v)]$ for a vertex $v\in V(G)$. Note that we have $\lk_{\cl(G)}\sigma=\cl(\lk_G\sigma)$, where each link should be understood in the appropriate sense. 

Finally, we define the \emph{join} $G\ast H$ of two graphs $G$ and $H$ as the disjoint union of $G$ and $H$ 
together with all the edges between $V(G)$ and $V(H)$. Once again we have $\cl(G\ast H)=\cl(G)\ast\cl(H)$, where on the right-hand side we have the simplicial join.

\medskip
The sole purpose of the next definition is to introduce a short substitute for the phrase ``$G$ is the $1$-skeleton of a $d$-dimensional flag weak pseudomanifold''.

\begin{definition}
\label{def:dlevel}
A graph $G$ will be called \emph{$d$-leveled} if it is the $1$-skeleton of a $d$-dimensional flag weak pseudomanifold. Explicitly,
\begin{itemize}
\item every maximal clique in $G$ has size $d+1$, and 
\item for every clique $\sigma$ of size $d$ the link $\lk_G\sigma$ consists of two isolated vertices.
\end{itemize}
\end{definition}
As a consequence, if $G$ is $d$-leveled and $\sigma$ is a clique in $G$ then $\lk_G\sigma$ is $(d-|\sigma|)$-leveled. In particular, in a $d$-leveled graph the link of every clique of size $d-1$ is a disjoint union of cycles. 

The next theorem is a straightforward equivalent formulation of Theorem~\ref{thm:weakpseudo}.
\begin{theorem}
\label{thm:dlevelbound}
For every $s\geq 1$ and every constant $C>0$ there is an $n_0=n_0(s,C)$ with the following property. If $G$ is a $(2s-1)$-leveled graph with $n\geq n_0$ vertices and with $k_{s+1}(G)\leq Cn^s$, then
$$|E(G)|\leq \frac{s-1}{2s}n^2+n.$$
\end{theorem}

\subsection*{Organization of the paper} In Section~\ref{section:extremal} we outline the proof of Theorem~\ref{thm:dlevelbound} and we introduce the toolbox of extremal graph theory. The main work is done in Section~\ref{section:dlevel}, where we apply the extremal techniques in the context of $(2s-1)$-leveled graphs. The proof of Theorem~\ref{thm:dlevelbound} appears at the end of that section. In Section~\ref{section:proofs} we present some further results and open problems.

The Appendix contains an alternative, very short proof of Corollary~\ref{thm:mainodd}.b), that is an upper bound in the case when $K$ is a flag $(2s-1)$-manifold. The reader interested only in flag spheres or flag manifolds can jump directly there.

\section{Extremal graph theory}
\label{section:extremal}

Let us outline the main idea behind the proof of Theorem~\ref{thm:dlevelbound}. Suppose, contrary to what we want to show, that $G$ satisfies the assumption of Theorem~\ref{thm:dlevelbound} but $|E(G)|>\frac{s-1}{2s}n^2+n$. Then $G$ is dense enough to exceed the Turan bound for the existence of $(s+1)$-cliques. However, the number $k_{s+1}(G)$ of those cliques is relatively small, namely $O(n^s)$, while a ``typical'' graph with this edge density ought to have $\Theta(n^{s+1})$ cliques of size $(s+1)$. This exceptional situation puts strong restrictions on the structure of $G$, which must be ``similar'', in a sense made precise in Theorem~\ref{thm:lovasz}, to the graph $K^s(\frac{n}{s}, \ldots,\frac{n}{s})$. It is a consequence of a phenomenon called \emph{supersaturation} \cite{ErdosSatur}, which is one of the basic principles of extremal graph theory.

At this point we start exploiting the extra property that $G$ is $(2s-1)$-leveled. This rigidifies the structure of $G$, up to the point of showing that it is either a join of $s$ cycles or it has even fewer edges than such a join. This approach, known as the \emph{stability method} and introduced in \cite{Nonexistent}, is relatively standard in extremal graph theory. In any case, we conclude that $|E(G)|\leq \frac{s-1}{2s}n^2+n$, contrary to the initial assumption. This part of the proof is carried out in Section~\ref{section:dlevel}. 

\medskip
The similarity to the complete $t$-partite graph, and more generally the similarity to the join of $t$ graphs, is best captured by the next definition.

\begin{definition}
\label{def:type}
We say a graph $G$ with a partition $V(G)=S_1\sqcup\cdots\sqcup S_t\sqcup X$ is of type $(t,\eta,C)$ if
\begin{itemize}
\item for every $i$, every $v\in S_i$ and every $j\neq i$ we have $\deg(v,S_j)\geq |S_j|(1-\eta)$,
\item $|X|\leq C$.
\end{itemize}
Moreover, we say such the partition is
\begin{itemize}
\item $m$-large if $|S_i|\geq m$ for all $i$,
\item $(n,\alpha)$-flat if $|V(G)|=n$ and for all $i$ we have
$$\frac{n}{t}(1-\alpha)\leq |S_i|\leq \frac{n}{t}(1+\alpha).$$
\end{itemize}
\end{definition}

If $G$ is a graph of type $(t,0,0)$ then $G=G[S_1]\ast \cdots\ast G[S_t]$. If $C\geq 0$ and $0\leq \eta\ll 1$ are constant then the graphs of type $(t,\eta,C)$ are ``almost'' joins, where ``almost'' refers to the possible existence of a constant number of exceptional vertices in $X$ and the absence of a small fraction of edges between the parts. One should think of $m$ as large and of $\alpha$ as very small, usually $0<\alpha\ll\eta$.

\medskip
Next we quote two results in extremal graph theory which will be needed later. The first one is a lower bound for the growth of the number of cliques in a graph whose number of edges passes the Turan bound.

\begin{theorem}
\label{thm:bolo}
For every $t\geq 1$ and $C\geq 0$ there exists $C'=C'(t,C)$ such that if $G$ is a graph with $n$ vertices and at least $\frac{t-1}{2t}n^2+C'n$ edges then $k_{t+1}(G)\geq Cn^t$.
\end{theorem}
\begin{proof}
This is a corollary of a theorem of Bollob\'as \cite{Bela}, which we now recall. For $n\in\nat$ and $m\in\er$ let $k_s(n,m)=\min\{k_s(H)~:~|V(H)|=n\ \textrm{and}\ |E(H)|\geq m\}$. By putting $p=2$, $r=t$, $q=t-1$ and $x=m$ in {\cite[Corollary 2]{Bela}} we obtain the inequality
\begin{equation}
\label{syfek}
k_{t+1}(n,m)\geq \frac{1}{(t+1)^t}(2tmn^{t-1}-(t-1)n^{t+1})\qquad \mathrm{for}\quad \frac{t-1}{2t}n^2\leq m\leq\frac{t}{2(t+1)}n^2
\end{equation}

Now note that it suffices to prove the theorem for $n\geq n_0=n_0(t,C)$. Small values of $n$ can be treated by later increasing $C'$. Define $C'=\frac{C(t+1)^t}{2t}$ and suppose that $n$ is large enough so that $\frac{t-1}{2t}n^2+C'n\leq\frac{t}{2(t+1)}n^2$. Then by \eqref{syfek} a graph $G$ with $|E(G)|\geq\frac{t-1}{2t}n^2+C'n$ satisfies
$$k_{t+1}(G)\geq \frac{2C'tn^{t}}{(t+1)^t}=Cn^t.$$
\end{proof}

\begin{remark}
\label{remark:reiher}
An asymptotically exact value of $k_s(n,m)$ can be found in the deep result of \cite{Reiher}, but its application would not lead to any substantial improvement of the bounds we get, other than tweaking some constants which are anyhow far from optimal, for instance in Theorem~\ref{thm:anydim}.
\end{remark}

\begin{remark}
\label{remark:cliquecounting}
Theorem~\ref{thm:bolo} immediately implies that solely due to clique counting, and without using the assumption that $G$ is $(2s-1)$-leveled, the upper bound in Theorem~\ref{thm:dlevelbound} holds with an accuracy of $O(n)$ edges. The exact linear term can be found thanks to the combinatorial properties of $G$.
\end{remark}

Next we quote the deep result of \cite{LS}, which is our main device for controlling the approximate structure of the graphs under consideration.

\begin{theorem}[Lov\'asz, Simonovits, {\cite[Theorem 2]{LS}}]
\label{thm:lovasz}
For every $s\geq 1$ and every constant $C>0$ there exist $\delta=\delta(s,C)$ and $C'=C'(s,C)$ with the following property. If $0<k<\delta n^2$ and $G$ is a graph with $n$ vertices for which
$$|E(G)|= \frac{s-1}{2s}n^2+k \quad\textrm{and}\quad k_{s+1}(G)\leq Ckn^{s-1}$$
then there are integers $n_1,\ldots,n_s$ such that 
$$\sum n_i=n\quad\textrm{and}\quad \frac{n}{s}-C'\sqrt{k}\leq n_i\leq \frac{n}{s}+ C'\sqrt{k}$$
and $G$ can be obtained from $K^s(n_1,\ldots,n_s)$ by changing (adding/removing) at most $C'k$ edges.
\end{theorem}

From this we get the following corollary, which adjusts the approximate structure so that it fits the framework of Definition~\ref{def:type}.

\begin{corollary}
\label{corollary:main}
For every $s\geq 1$, $\eta>0$, $\alpha>0$ and $C>0$ there exist $n_0=n_0(s,\eta,\alpha,C)$ and $C'=C'(s,\eta,\alpha,C)$ such that every graph $G$ with $n\geq n_0$ vertices for which
$$|E(G)|>\frac{s-1}{2s}n^2+n \quad\textrm{and}\quad k_{s+1}(G)\leq Cn^{s}$$
is $(n,\alpha)$-flat of type $(s,\eta,C')$.
\end{corollary}
\begin{proof}
Suppose $n$ is sufficiently large and let $|E(G)|=\frac{s-1}{2s}n^2+k$. By Theorem~\ref{thm:bolo} we get $n\leq k\leq C_1n$ for some constant $C_1=C_1(s,C)$. Moreover, we have
$$k_{s+1}(G)\leq Cn^s\leq Ckn^{s-1}$$
so for sufficiently large $n$ the graph $G$ satisfies the assumptions of Theorem~\ref{thm:lovasz}. We obtain a decomposition $V(G)=A_1\sqcup \cdots\sqcup A_s$ which satisfies, for a suitable constant $C_2=C_2(s,C_1)$, the conditions
\begin{eqnarray}
\label{eq:loc1} \frac{n}{s}-C_2\sqrt{n}\leq|A_i|\leq\frac{n}{s}+C_2\sqrt{n} & \textrm{for all}\ i,\\
\label{eq:loc2} |\overline{E}(G[A_i,A_j])|\leq C_2n & \textrm{for all}\ i,j,
\end{eqnarray}
where $|\overline{E}(G[A_i,A_j])|$ is the number of edges missing between the parts $A_i$ and $A_j$ in $G$.

For $i\neq j$ define
$$X_{i,j}=\{v\in A_i~:~\deg(v,A_j)\leq \frac{n}{s}(1-\frac12\eta)\}.$$
Note that by \eqref{eq:loc1} for large $n$ we have
$$|\overline{E}(G[A_i,A_j])|\geq |X_{i,j}|(|A_j|-\frac{n}{s}(1-\frac12\eta))\geq|X_{i,j}|\cdot\frac{\eta n}{3s}$$
so \eqref{eq:loc2} gives $|X_{i,j}|\leq C_3$ for $C_3=3sC_2\eta^{-1}$. Set $X=\bigcup_{i,j}X_{i,j}$, $C'=s^2C_3$ and $S_i=A_i\setminus X$. We clearly have $|X|\leq C'$. For every $i$, every $v\in S_i$ and all $j\neq i$ we get
$$|S_i|\geq |A_i|-C'\geq \frac{n}{s}-C_2\sqrt{n}-C'\geq \frac{n}{s}(1-\alpha),$$
$$|S_i|\leq |A_i|\leq \frac{n}{s}+C_2\sqrt{n}\leq \frac{n}{s}(1+\alpha),$$
$$\deg(v,S_j)\geq\deg(v,A_j)-C'\geq\frac{n}{s}(1-\frac12\eta)-C'\geq(\frac{n}{s}+C_2\sqrt{n})(1-\eta) \geq|S_j|(1-\eta).$$
It follows that the decomposition $V(G)=S_1\sqcup\cdots\sqcup S_s\sqcup X$ makes $G$ an $(n,\alpha)$-flat graph of type $(s,\eta,C')$.
\end{proof}

For convenience we collect some facts about graphs of type $(t,\eta,C)$, all of which follow easily from the definition. In the next section we are going to use them freely without explicit reference.
\begin{fact}
\label{fact}
Suppose $G$ is a graph of type $(t,\eta,C)$ with a decomposition $V(G)=S_1\sqcup\cdots\sqcup S_t\sqcup X$.
\begin{itemize}
\item[a)] If $0<\beta\leq 1$ and $T_i\subset S_i$, $X'\subset X$ with $|T_i|\geq \beta|S_i|$ for all $i$ then the graph $G[\bigcup T_i\cup X']$ is of type $(t, \eta\beta^{-1}, C)$.

\item[b)] If $1\leq k\leq t-1$ and $P\subset \bigcup_{i=1}^kS_i$ then 
$$|S_{k+1}\cap\bigcap_{v\in P}N_G(v)|\geq (1-\eta|P|)|S_{k+1}|.$$

\item[c)] If $\eta<\frac{1}{t}$ then $G[\bigcup S_i]$ has a clique $\sigma$ with $|\sigma\cap S_i|=1$ for all $i$.

\item[d)] If $1\leq k\leq t-1$ and $\sigma\subset \bigcup_{i=1}^kS_i$ is a maximal clique in $G[\bigcup_{i=1}^kS_i]$ then $\lk_G\sigma$ is a graph of type $(t-k,\ \eta(1-\eta|\sigma|)^{-1},\ C)$, which is $((1-\eta|\sigma|)m)$-large if $G$ was $m$-large.
\end{itemize}
\end{fact}

\section{Dense $d$-leveled graphs}
\label{section:dlevel}

In this section we study graphs which are simultaneously of type $(t,\eta,C)$ and $d$-leveled for some $d$. The main technical result of this section is Theorem~\ref{thm:denseshit}, which gives an upper bound on the number of edges in such graphs when $d=2t-1$. Theorem~\ref{thm:dlevelbound} then follows easily, and it is proved at the end of this section. First we need two auxiliary results.

\begin{lemma}
\label{lem:indep}
Suppose that $d\geq 0$ and $G$ is a $d$-leveled graph with $V(G)=I\sqcup X$ where $I$ is an independent set in $G$. Then 
$$|I|\leq 2|X|^d.$$
\end{lemma}
\begin{proof}
Every vertex $v\in I$ forms a $(d+1)$-clique with at least one $d$-element subset $\sigma$ of $X$. On the other hand, every $d$-element clique $\sigma\subset X$ belongs to at most two such $(d+1)$-cliques. It follows that
$|I|\leq 2{|X|\choose d}\leq 2|X|^d$ as required.
\end{proof}

\begin{notdef}
For the remainder of the paper we fix small constants $\eta_t, \alpha_t$ for $t=1,2\ldots$ subject to the conditions
$$\eta_t<\frac{1}{100t}\ \textrm{for all}\ t\geq 1,\quad \eta_t(1-2t\eta_t)^{-1}<\eta_{t-1}\ \textrm{for all}\  t\geq 2$$
and $\alpha_t<\frac{\eta_t}{10t}$.
\end{notdef}

\begin{proposition}
\label{prop:noshit}
Let $t\geq 1$. For every $C\geq 0$ there is an $m=m(t,C)$ with the following property. If $G$ is an $m$-large graph of type $(t,\eta_t,C)$ then $G$ is not $d$-leveled for any $d\leq 2t-2$.
\end{proposition}
\begin{proof}
We prove the statement by induction on $t$. If $t=1$ and $m=3$ then $G$ cannot be $0$-leveled, so we are done. 

Let $t\geq 2$ and assume that $G$ is $m$-large of type $(t,\eta_t,C)$ for sufficiently large $m$. Suppose, contrary to what we need to show, that $G$ is $d$-leveled for some $d\leq 2t-2$. By Fact~\ref{fact} $G$ has cliques of size $t$ so $d+1\geq t$.

If some $G[S_i]$, say $G[S_1]$, contains a maximal clique $\sigma$ of size $c\geq 2$ then we have $2\leq c\leq d+1\leq 2t-1$. The link $\lk_G\sigma$ is a $(d-c)$-leveled graph of type $(t-1, \eta_{t}(1-2t\eta_t)^{-1} ,C)$, hence also of type $(t-1,\eta_{t-1},C)$, and it is $(1-2t\eta_t)m$-large. For large enough $m$ induction yields $d-c\geq 2(t-1)-1$ i.e. $d\geq 2t+c-3\geq 2t-1$, contrary to our hypothesis.

It means that each $G[S_i]$ is edge-less. In this case choose a clique $\sigma$ of size $t-1$ containing exactly one vertex from each of $S_1,\ldots,S_{t-1}$. The link $\lk_G\sigma$ is a $(d-t+1)$-leveled graph contained in $S_t\sqcup X$, and the part contained in $S_t$ is an independent set of size at least $(1-t\eta_t)m$.  Lemma~\ref{lem:indep} gives
$$(1-t\eta_t)m\leq |V(\lk_G\sigma)\cap S_t|\leq 2|X|^d\leq 2C^d$$
which is a contradiction for sufficiently large $m$.
\end{proof}

Now we come to the main result which bounds the number of edges among those $(2t-1)$-leveled graphs which are  similar to a join of $t$ almost equal parts in the sense of Definition~\ref{def:type}. 

\begin{theorem}
\label{thm:denseshit}
Let $t\geq 1$. For every $C\geq 0$ there is an $n_0=n_0(t,C)$ with the following property. Suppose $G$ is an $(n,\alpha_t)$-flat graph of type $(t,\frac{\eta_t}{12},C)$ with $n\geq n_0$ vertices. If $G$ is $(2t-1)$-leveled then
\begin{equation*}
\label{eq:thmdenseshit}
|E(G)|\leq \frac{t-1}{2t}n^2+n.
\end{equation*}
\end{theorem}
\begin{proof}

If $t=1$ the result is obvious, so suppose $t\geq 2$.

For $v\in X$ let $l(v)=|\{i~:~\deg(v,S_i)\leq \frac{n}{t}(1-\frac{\eta_t}{12})\}|$. If $l(v)=0$ then we move $v$ to any $S_i$ and if $l(v)=1$ then we move $v$ to the only set $S_i$ for which $\deg(v,S_i)\leq \frac{n}{t}(1-\frac{\eta_t}{12})$. We iterate this operation as long as possible. Since we moved only at most $C$ vertices, we can assume that for large enough $n$ with the new partition $G$ is
$$(n,2\alpha_t)\textrm{-flat of type}\ (t,\frac{\eta_t}{6},C)$$
and that every vertex $v\in X$ satisfies $l(v)\geq 2$. From now on we are going to use this new partition and for simplicity we set
$$\alpha:=2\alpha_t,\ \eta:=\frac{\eta_t}{6}.$$

The proof will be split into a series of claims. Note that $G$ is $\frac{n}{2t}$-large.

\begin{claim}
\label{claim1}
For every $i$ the graph $G[S_i]$ is triangle-free.
\end{claim}
\begin{proof}
Let $\sigma$ be a maximal clique in $G[S_1]$ and let $c=|\sigma|$. Then $c\leq 2t-1$ because $G$ is $(2t-1)$-leveled and $\lk_G\sigma$ is nonempty. We have that $\lk_G\sigma$ is a $(2t-1-c)$-leveled graph of type $(t-1,\eta(1-2t\eta)^{-1},C)$, and so of type $(t-1,\eta_{t-1},C)$, which is $\frac{n}{2t}(1-2t\eta)$-large. By Proposition~\ref{prop:noshit} for sufficiently large $n$ we have $2t-1-c\geq 2(t-1)-1$, hence $c\leq 2$.
\end{proof}

As a consequence every clique in $G[\bigcup S_i]$ has size at most $2t$, and so every maximal clique in $G[\bigcup S_i]$ has nonempty intersection with each $S_i$.

\begin{claim}
\label{claim2}
There exists a clique in $G[\bigcup S_i]$ of size $2t$.
\end{claim}
\begin{proof}
Suppose that the largest clique $\sigma$ in $G[\bigcup S_i]$ has size $c<2t$. By Claim~\ref{claim1} and the last remark there exist an index $i$ and a vertex $v\in S_i$ such that $\sigma\cap S_i=\{v\}$. In particular, the link $\lk_G(\sigma\setminus\{v\})$ does not contain any edge of $G[S_i]$ (otherwise we would have obtained a clique of size $c+1$). Also, observe that $\sigma\setminus\{v\}$ is a maximal clique in $G[\bigcup_{j\neq i}S_j]$, since otherwise its superclique $\tau$, extended by some vertex of $S_i$, would have size $c+1$. It means that $\lk_G(\sigma\setminus\{v\})$ is a $(2t-c)$-leveled graph contained in $S_i\sqcup X$ and the part contained in $S_i$ is an independent set of size $k\geq\frac{n}{2t}(1-2t\eta)$. For large $n$ this is impossible because of Lemma~\ref{lem:indep}, as in the proof of Proposition~\ref{prop:noshit}.
\end{proof}

\begin{claim}
\label{claim3}
For every $i$ if $Y\subseteq S_i$ with $|Y|\geq \frac{5}{8}\cdot\frac{n}{t}$ then $G[Y]$ contains an edge.
\end{claim}
\begin{proof}
It suffices to consider $Y\subseteq S_1$. Let $\sigma$ be some maximal clique in $G[\bigcup S_i]$ of size $2t$ and let $\sigma\cap S_1=\{x,y\}$ for some edge $xy\in E(G)$. Then $H=\lk_G(\sigma\setminus\{x,y\})$ is a $1$-leveled graph contained in $S_1\sqcup X$ and the part contained in $S_1$ has $k\geq|S_1|(1-2t\eta)$ vertices. We see that $H$ occupies almost all of $S_1\sqcup X$ while $Y$ occupies much more than a half of that set, in particular 
$$|Y\cap V(H)|\geq |S_1|(1-2t\eta)-(|S_1|-\frac58\cdot\frac{n}{t})\geq \frac{1}{2}(|S_1|+C)\geq |V(H)|/2.$$ Since $H$ is a union of cycles, that implies that $G[Y\cap V(H)]$ contains an edge.
\end{proof}

\begin{claim}
\label{claim4}
For every $i$ the graph $G[S_i]$ has maximum degree $2$.
\end{claim}
\begin{proof}
Suppose $v\in S_1$ has three neighbors $u_1,u_2,u_3\in S_1$. We are going to inductively construct cliques $\sigma_k\subset \bigcup_{i=2}^kS_i$ for $k=1,\ldots,t$ such that $\sigma_k\subset\sigma_{k+1}$, $|\sigma_k|=2(k-1)$ and $v,u_1,u_2,u_3\in\lk_G\sigma_k$ for all $k$. This gives a contradiction since $u_1,u_2,u_3\in\lk_G(\sigma_t\cup\{v\})$, contrary to the fact that $\lk_G(\sigma_t\cup\{v\})$ is $0$-leveled.

We set $\sigma_1=\emptyset$. Suppose $\sigma_k$ was constructed and let $P=\{v,u_1,u_2,u_3\}\cup\sigma_k$. Then the set $Y_{k+1}=S_{k+1}\cap\bigcap_{x\in P}N(x)$ contains at least $(1-2t\eta)|S_{k+1}|$ vertices. By the previous claim there is an edge $a_{k+1}b_{k+1}\in G[Y_{k+1}]$ and we set $\sigma_{k+1}=\sigma_k\cup\{a_{k+1},b_{k+1}\}$.
\end{proof}

\begin{claim}
\label{claim5}
If $v\in X$ then $\deg(v,\bigcup S_i)\leq n(1-\frac{1}{t}-\frac{\eta}{3t})$.
\end{claim}
\begin{proof}

Let $X'\subset X$ be the set of all vertices $v$ such that $\deg(v,S_i)\leq 2$ for some $i$. We first prove the claim for $v\in X'$. Assume without loss of generality that $\deg(v,S_1)\leq 2$. Since $l(v)\geq 2$, there exists $j\neq 1$ such that $\deg(v,S_j)\leq \frac{n}{t}(1-\frac12\eta)$. Then
$$\deg(v,\bigcup S_i)\leq (t-2)\frac{n}{t}(1+\alpha)+\frac{n}{t}(1-\frac12\eta)+2\leq n(1-\frac{1}{t}-\frac{\eta}{3t})$$
for sufficiently large $n$.

Next consider the vertices $v\in X\setminus X'$. For every such $v$ and every $i$ we have $\deg(v,S_i)\geq 3$. Define $k(v)=|\{i~:~\deg(v,S_i)\leq \frac35\cdot\frac{n}{t}\}|$. The proof depends on the value of $k(v)$.

\subsection*{The case $k(v)\geq 3$} We have 
$$\deg(v,\bigcup S_i)\leq (t-3)\frac{n}{t}(1+\alpha)+3\cdot\frac35\cdot\frac{n}{t}\leq n(1-\frac{11}{10}\cdot\frac{1}{t}).$$

\subsection*{The case $k(v)=2$} Without loss of generality let $i=1,2$ be the indices for which $\deg(v,S_i)\leq \frac35\cdot\frac{n}{t}$. If for both of them we have $\deg(v,S_i)\geq \frac{1}{6}\cdot\frac{n}{t}$ then $\lk_Gv$ is a $(2t-2)$-leveled graph, which is $\frac{n}{6t}$-large of type $(t,6\eta,C)$, hence also of type $(t,\eta_t,C)$. For large $n$ this is not possible by Proposition~\ref{prop:noshit}. It means that for at least one of $i=1,2$ we have $\deg(v,S_i)\leq \frac{1}{6}\cdot\frac{n}{t}$. Then

$$\deg(v,\bigcup S_i)\leq (t-2)\frac{n}{t}(1+\alpha)+\frac35\cdot\frac{n}{t}+\frac{1}{6}\cdot\frac{n}{t}\leq n(1-\frac{6}{5}\cdot\frac{1}{t}).$$

\subsection*{The case $k(v)=1$} Assume $i=1$ is the only index with $\deg(v,S_i)\leq \frac35\cdot\frac{n}{t}$. Let $u_1,u_2,u_3\in S_1$ be any three neighbors of $v$. As in the proof of Claim~\ref{claim4} we can now construct a clique $\sigma=\{a_2,b_2,\ldots,a_t,b_t\}$ with $a_k,b_k\in S_k$ for $k=2,\ldots,t$ and such that $v,u_1,u_2,u_3\in\lk_G\sigma$. (The set $Y_{k+1}\subset S_{k+1}$ constructed in that proof now has at least $\frac{n}{t}(\frac35-4t\eta)$ vertices). But then $\lk_G(\sigma\cup\{v\})$ is a $0$-leveled graph which contains $3$ vertices $u_1,u_2,u_3$. That is a contradiction.

\subsection*{The case $k(v)=0$} Then $\lk_Gv$ is a $(2t-2)$-level graph of type $(t,\frac53\eta,C)$, hence also of type $(t,\eta_t,C)$, which is $\frac{3n}{5t}$-large. For large $n$ that is impossible by Proposition~\ref{prop:noshit}.
\end{proof}

Now we can count the number of edges in $G$ to show the desired bound. Let $x=|X|\leq C$. By Claims \ref{claim4} and \ref{claim5} we have
\begin{equation}
\label{eqn:totale}
\begin{aligned}
|E(G)|&\leq {t\choose 2}(\frac{n-x}{t})^2 + (n-x) + xn(1-\frac{1}{t}-\frac{\eta}{2t})+{x\choose 2}\\
&\leq\frac{t-1}{2t}n^2+n-x(\frac{\eta}{2t}n-C\cdot\frac{2t-1}{2t}).
\end{aligned}
\end{equation}
This completes the proof of Theorem~\ref{thm:denseshit}.
\end{proof}

\bigskip
We can now prove the main result of the paper.

\begin{proof}[Proof of Theorem~\ref{thm:dlevelbound}]
Suppose, contrary to what we want to prove, that $|E(G)|> \frac{s-1}{2s}n^2+n$. Then the assumptions of Corollary~\ref{corollary:main} are satisfied, hence for sufficiently large $n$ the graph $G$ is $(n,\alpha_s)$-flat of type $(s,\frac{\eta_s}{12},C')$ for some constant $C'=C'(s,C)$. However, since $G$ is $(2s-1)$-leveled, by Theorem~\ref{thm:denseshit} it must satisfy $|E(G)|\leq \frac{s-1}{2s}n^2+n$. This contradiction ends the proof.
\end{proof}

\section{Related results and problems}
\label{section:proofs}

In this section we discuss some related problems and possible relaxations of the assumptions of Theorem~\ref{thm:weakpseudo}.

As noted in Remark~\ref{remark:cliquecounting}, already the basic clique counting argument provides an upper bound for the number of edges which is optimal up to an $O(n)$ term. The advantage of that argument is that it works for flag weak pseudomanifolds of arbitrary dimension. More precisely, one immediately obtains the following result.

\begin{proposition}
\label{thm:anydim}
Let $d\geq 1$ and $s=\lfloor\frac{d+1}{2}\rfloor$. For every constant $C>0$ there is a $C'=C'(d,C)$ such that the inequality 
\begin{equation}\label{eq:withC}f_1(K)\leq \frac{s-1}{2s}f_0(K)^2+C'\cdot f_0(K)\end{equation}
holds for each of the following classes of complexes $K$:
\begin{itemize}
\item[a)] $d$-dimensional flag complexes with $f_s(K)\leq Cf_0(K)^s$,
\item[b)] $d$-dimensional flag complexes which satisfy the middle Dehn-Sommerville equation ($h_s(K)=h_{s+1}(K)$ when $d=2s$ or $h_{s-1}(K)=h_{s+1}(K)$ when $d=2s-1$),
\item[c)] flag simplicial $d$-manifolds whose Euler characteristic satisfies $\chi(K)\leq Cf_0(K)^s$,
\item[d)] flag $d$-Gorenstein$^*$ complexes.
\end{itemize}
\end{proposition}
\begin{proof}
Part a) follows directly from Theorem~\ref{thm:bolo}. Regardless of the parity of $d$, the middle Dehn-Sommerville equation has the form $f_s(K)=\sum_{i=-1}^{s-1}a_{s,i}f_i(K)$, hence $f_s(K)\leq (\sum_i|a_{s,i}|)\cdot f_0(K)^s$ and therefore b) follows from a). Part b) immediately implies d). Finally for any $d$-manifold we have the generalized Dehn-Sommerville equations of Klee \cite{Klee}, namely $$h_{d+1-i}(K)-h_i(K)=(-1)^i{d+1\choose i}(\chi(K)-\chi(S^{d})),$$ and the middle of them again yields $f_s(K)\leq C_1 f_0(K)^s$ for some constant $C_1$, thus  reducing part c) to a).
\end{proof}

As in the odd-dimensional case, for spheres one is expecting an exact upper bound of the following form.
\begin{conjecture}[{\cite[Conj. 6.3]{NevoPet}}]
\label{conj:even}
If $s\geq 1$ and $K$ is a flag triangulation of $S^{2s}$ then
\begin{equation}\label{eq:s2s}f_1(K)\leq \frac{s-1}{2s}f_0(K)^2+(1+\frac{2}{s})f_0(K)-(4+\frac{2}{s}).\end{equation}
\end{conjecture}
In this case equality is achieved by the join of $s-1$ cycles of length $k$ with an arbitrary $(k+2)$-vertex flag triangulation of $S^2$, assuming that $f_0(K)=sk+2$.

Because the Euler characteristic of a $2s$-dimensional closed manifold is $O(f_0(K)^{s+1})$ (by Poincar\'e duality), but not necessarily $O(f_0(K)^s)$, the bounds of Proposition~\ref{thm:anydim} are not automatically satisfied by \emph{all} closed flag $2s$-manifolds. In fact for arbitrary $\varepsilon>0$ there are flag $2$-manifolds $K$ with  $f_1(K)=\Omega(f_0(K)^{2-\varepsilon})$, which matches the upper bound of $f_1(K)=o(f_0(K)^2)$ (see below). The situation in higher even dimensions is not known.

\begin{problem}
\label{prob:evenmani}
When $s\geq 2$, what is the maximum of $f_1(K)$ among all closed flag $2s$-manifolds $K$ with a fixed number of vertices?
\end{problem}

\medskip
One can also ask to what extent the assumption $f_s(K)\leq Cf_0(K)^s$ in Theorem~\ref{thm:weakpseudo} is needed. We will show that there are $(2s-1)$-dimensional flag weak pseudomanifolds which do not satisfy this condition. For this we need a result of \cite{Cycle}, that for any $\varepsilon>0$ there exist $n$-vertex graphs $\mathcal{H}_n^\varepsilon$, with arbitrarily large $n$, such that $|E(\mathcal{H}_n^\varepsilon)|=\Omega(n^{2-\varepsilon})$ and such that for every vertex $v$ the link $\lk_{\mathcal{H}_n^\varepsilon}v$ is a cycle of length at least $4$. In particular, $\mathcal{H}_n^\varepsilon$ is $2$-leveled and $K=\cl(\mathcal{H}_n^\varepsilon)$ is a flag $2$-manifold with $\Omega(f_0(K)^{2-\varepsilon})$ edges. Now the join $G_{2n}^\varepsilon=\mathcal{H}_n^\varepsilon\ast \mathcal{H}_n^\varepsilon$ is a family of $5$-leveled graphs with $k_4(G_{2n}^\varepsilon)=\Omega(n^{4-2\varepsilon})$. That means that a $5$-dimensional flag weak pseudomanifold $K$ can have $f_3(K)=\Omega(f_0(K)^{4-2\varepsilon})$.

Still, even though this approach fails, it is possible that the bound on the number of edges holds.
\begin{problem}
\label{prob2}
Is it true that any $(2s-1)$-leveled graph with $n$ vertices has at most $\frac{s-1}{2s}n^2+n$ edges?
\end{problem}
The answer to this problem is obviously positive when $s=1$.

\medskip
In even dimension $2s$ we cannot expect a similar bound since the join of $\mathcal{H}_{n/s}^\varepsilon$ with $s-1$ copies of the $n/s$-vertex cycle is a $2s$-leveled graph with $\frac{s-1}{2s}n^2+\Omega(n^{2-\varepsilon})$ edges. However, one can ask the following.
\begin{problem}
\label{prob3}
Is it true that any $2s$-leveled graph with $n$ vertices has at most $\frac{s-1}{2s}n^2+o(n^2)$ edges?
\end{problem}
Again, the answer is positive when $s=1$. It follows from \cite[Theorem 1]{Cycles2}, which says that a graph in which every edge belongs to at least one, but at most a fixed number of triangles, can only have $o(n^2)$ edges.

\subsection*{Acknowledgement} The author thanks Jan Hladk\'y and Eran Nevo for support and comments.



\section{Appendix: Odd-dimensional manifolds}
\label{section:oddmanifolds}
\setcounter{claim}{0}

The purpose of this section is to give an independent quick proof of Corollary~\ref{thm:mainodd}
in the case when $K$ is a flag simplicial $(2s-1)$-manifold. 

\begin{proof}[Proof of Corollary~\ref{thm:mainodd}, case b)]
Let $K$ be any flag simplicial $(2s-1)$-manifold with $n$ vertices. Let $G$ be the $1$-skeleton of $K$. For any vertex $v$ the link $\lk_Kv$ is a flag simplicial $2(s-1)$-sphere. 

A classical theorem of van Kampen and Flores \cite{vanKamp,Flores} (see also \cite[Sect 2.4]{Wagner}), which generalizes the standard result about planar graphs, states that for all $s\geq 1$ the space $\cl(K^s(3,\ldots,3))$ does not embed in $S^{2(s-1)}$. It follows that for all $v$ the link $\lk_Gv$ does not contain $K^s(3,\ldots,3)$ as a subgraph. As a consequence, $G$ does not contain $K^{s+1}(1,3,\ldots,3)$ as a subgraph.

By the main theorem of \cite{ErdosProduct} a sufficiently large graph which does not contain $K^{s+1}(1,3,\ldots,3)$ has at most $\frac{s-1}{2s}n^2+n$ edges. That ends the proof.
\end{proof}

The theorem of \cite{ErdosProduct} referred to in the above proof is more general. For any tuple $r_1\leq r_2\leq\cdots\leq r_{s+1}$ with $r_1\in\{1,2,3\}$ it characterizes, for sufficiently large $n$, all $n$-vertex graphs without a $K^{s+1}(r_1,\ldots,r_{s+1})$ which have the maximal possible number of edges. For $r_1=1$, $r_2=\cdots=r_{s+1}=3$ that characterization reduces to the claim that every such extremal graph is a join of $s$ graphs, each of maximal degree $2$. From there the bound on the number of edges immediately follows.

For $r_1=1$ that result is only stated, but not proved in \cite{ErdosProduct}. For the proof the reader is referred to \cite{Nonexistent}, which is not easily accessible. However, an independent proof can be obtained using the methods of this paper. The starting point is the stable version of the Erd\"os-Stone theorem \cite[p.184]{ErdosSatur}, which guarantees that an extremal $K^{s+1}(1,3,\ldots,3)$-free graph can be obtained from $K^s(\frac{n}{s},\ldots,\frac{n}{s})$ by changing $o(n^2)$ edges.

It is very likely that the characterization of extremal $K^{s+1}(1,3,\ldots,3)$-free graphs holds for \emph{all} values of $n$. This would imply that the inequality \eqref{eq:thmmain} holds for all, not just sufficiently large, flag simplicial $(2s-1)$-manifolds.


\begin{thebibliography}{99}
\bibitem{Idiots} M.Adamaszek, J. Hladk\'y, \textit{Dense flag triangulations of $3$-manifolds via extremal graph theory}, preprint

\bibitem{Bela} B.Bollob\'as, \textit{On complete subgraphs of different orders}, Math. Proc. Cambridge Phil. Soc. 79 (1976), 19-24

\bibitem{Cycles2} L.H.Clark, R.C.Entringer, J.E.McCanna, L.A.Sz\'ekely, \textit{Extremal problems for local properties of graphs}, Australasian J. Comb. 4 (1991), 25-31

\bibitem{CD} R.Charney, M.Davis, \textit{The {E}uler characteristic of a non-positively curved, piecewise {E}uclidean manifold}, Pacific J. Math. 171 (1995), 117-137

\bibitem{DOkun} M.Davis, B.Okun, \textit{Vanishing theorems and conjectures for the $\ell^2$-homology of right-angled Coxeter groups}, Geometry and Topology 5 (2001), 7-74

\bibitem{ErdosSatur} P.Erd\"os, M.Simonovits, \textit{Supersaturated graphs and hypergraphs}, Combinatorica, 3 (1983), 181-192

\bibitem{ErdosProduct} P.Erd\"os, M.Simonovits, \textit{An extremal graph problem}, Acta Math. Acad. Sci. Hung. 22 (1971), 275-282

\bibitem{Flores} A.Flores, \textit{\"Uber die Existenz $n$-dimensionaler Komplexe die nicht im den $\mathbb{R}^{2n}$  topologisch einbettar sind}, Ergeb. Math. Kolloq. 5 (1933), 17--24

\bibitem{Andy} A.Frohmader, \textit{Face vectors of flag complexes}, Israel J. Math. 164 (2008), 153-164

\bibitem{Gal} S.R.Gal, \textit{Real root conjecture fails for five- and higher-dimensional spheres}, Discrete Comput. Geom. 34 (2005), 269-284


\bibitem{vanKamp} E.R. van Kampen, \textit{Komplexe in euklidischen R\"aumen}, Abh. Math. Sem. Univ. Hamburg (1932), 9:72-7

\bibitem{Klee} V.Klee, \textit{A combinatorial analogue of Poincar\'e's duality theorem}, Canadian J. Math. 16 (1964), 517-531

\bibitem{LS} L.Lov\'asz, M.Simonovits, \textit{On the number of complete subgraphs of a graph II}, Studies in Pure Mathematics (1983), 459-495

\bibitem{MurNevo} S.Murai, E.Nevo, \textit{The flag $f$-vectors of {G}orenstein$^*$ order complexes of dimension $3$}, preprint 

\bibitem{NevoPet} E.Nevo, T.K.Petersen, \textit{On $\gamma$-vectors satisfying the {K}ruskal-{K}atona inequalities}, Discrete Comput. Geom. 45 (2011), 503-521

\bibitem{NevoPetTenner} E.Nevo, T.K.Petersen, B.E.Tenner, \textit{The $\gamma$-vector of a barycentric subdivision}, J. Comb. Theory Ser. A 118 (2011), 1346-1380

\bibitem{Reiher} C.Reiher, \textit{The clique density theorem}, preprint

\bibitem{Cycle} \'A.Seress, T.Szab\'o, \textit{Dense graphs with cycle neighborhoods}, J. Comb. Theory B 63 (2) (1995), 281-293

\bibitem{Nonexistent} M.Simonovits, \textit{A method for solving extremal problems in graph theory, stability problems}, Theory of Graphs (Proc. Colloq., Tihany, 1966), 279-319, Academic Press, NY, 1968

\bibitem{Wagner} U.Wagner, \textit{Minors in Random and Expanding Hypergraphs}, Proc. 27th Annual ACM Symposium on Computational Geometry (SoCG), 2011, 351-360. 

\end{thebibliography}
\end{document}